\documentclass[12pt]{amsart}

\usepackage[matrix,arrow,curve,frame]{xy}
\usepackage{amsmath,amsthm,amssymb,enumerate}
\usepackage{latexsym}
\usepackage{amscd}
\usepackage[backref]{hyperref}
\usepackage{euscript}
\usepackage{fullpage}
\usepackage{color}
\usepackage{tikz}
\usepackage{tikz-cd}
\usepackage[all]{xypic}
\usepackage{mathptmx,enumitem,cite,array}
\usepackage[dvips]{geometry}

\newtheorem{theorem}{Theorem}[section]

\newtheorem{definition}[theorem]{Definition}
\newtheorem{example}[theorem]{Example}
\newtheorem{remark}[theorem]{Remark}

\theoremstyle{example}
\theoremstyle{remark}

\numberwithin{equation}{section}

\newcommand{\beq}{\begin{equation}}
\newcommand{\eeq}{\end{equation}}
\DeclareMathOperator{\ind}{index}

\newcommand{\CC}{\mathbb{C}}
\newcommand{\HH}{\mathbb{H}}
\newcommand{\QQ}{\mathbb{Q}}
\newcommand{\bS}{\mathbb{S}}

\newcommand{\bbZ}{\mathbb{Z}}
\newcommand{\cP}{\mathcal{P}}
\newcommand\CIc{{\mathcal{C}}^{\infty}_c}
\newcommand\CI{{\mathcal{C}}^{\infty}}
\newcommand{\F}{\mathcal{F}}

\DeclareMathOperator{\Hom}{Hom}

\DeclareMathOperator{\tr}{tr}
\DeclareMathOperator{\Tr}{Tr}

\newcommand{\C}{\mathbb{C}}
\newcommand{\R}{\mathbb{R}}
\newcommand{\Z}{\mathbb{Z}}

\newcommand{\Q}{\mathbb{Q}}
\newcommand{\GCh}{\textsf{GCh}}

\newcommand{\Ch}{\mathrm{Ch}}

\renewcommand{\HH}{\mathcal{H}}

\newcommand{\cA}{\mathcal A}

\newcommand{\cS}{\mathcal{S}}

\newcommand{\dirac}{\partial\!\!\!\!\!\slash}

\newcommand {\be}{\begin{equation}}
\newcommand {\ee}{\end{equation}}
\newcommand{\h}{\begin{eqnarray*}}
\newcommand{\e}{\end{eqnarray*}}

\begin{document}


\title
{Projective elliptic genera and elliptic pseudodifferential genera}

\author{Fei Han}
\address{Department of Mathematics,
National University of Singapore, Singapore 119076}
\email{mathanf@nus.edu.sg}

 \author{Varghese Mathai}
\address{School of Mathematical Sciences,
University of Adelaide, Adelaide 5005, Australia}
\email{mathai.varghese@adelaide.edu.au}

\subjclass[2010]{Primary 58J26, 58J40, Secondary 55P35, 11F55}
\keywords{Projective elliptic genera, projective elliptic pseudodifferential genera, graded twisted Chern character, modularity, Schur functors}
\date{}

\maketitle

\begin{abstract}
In this paper, we construct for the first time the projective elliptic genera for a compact oriented manifold equipped with a projective complex vector bundle. Such projective elliptic genera are rational $q$-series that have topological definition and also have analytic interpretation via the fractional index theorem in \cite{MMS06} without requiring spin condition. We prove the modularity properties of these projective elliptic genera. As an application, we construct elliptic pseudodifferential genera for any elliptic pseudodifferential operator. This suggests the existence of putative  $S^1$-equivariant  elliptic pseudodifferential operators on loop space whose equivariant indices are elliptic pseudodifferential genera.

\end{abstract}

\tableofcontents


\section*{Introduction}

In 1980's, Witten studied two-dimensional quantum field theories and the index of Dirac operator in free loop spaces. In \cite{W87}, Witten argued that the partition function of a type II superstring as a function depending on the modulus of the worldsheet elliptic curve, is an elliptic genus. In \cite{W86}, Witten derived a series of twisted Dirac operators from the free loop space $LZ$ on a compact spin manifold $Z$. The elliptic genera constructed by Landweber-Stong \cite{LS88} and Ochanine \cite{O87} in a topological way turn out to be the indices of these elliptic operators. Motivated by physics, Witten conjectured that these elliptic operators should be rigid. The Witten conjecture was first proved by Taubes \cite{T89} and Bott-Taubes \cite{BT89}. In \cite{Liu96}, using the modular invariance property, Liu presented a simple and unified proof of the Witten conjecture. A useful reference in this area is the book by Hirzebruch, Berger and Jung \cite{HBJ}.  We also mention the recent generalisation of these genera by the authors in \cite{HM18} to noncompact manifolds with noncompact almost connected Lie groups acting properly and cocompactly. Let us be more precise as follows. 

Let $Z$ be a $4r$ dimensional compact smooth spin manifold and $V$ be a rank $2l$ spin vector bundle over $Z$.
As in \cite{W86}, let
\be \label{reducedwittenbundle} \Theta(\widetilde{T_\CC Z})=\bigotimes_{n=1}^\infty
S_{q^n}(T_\CC Z-\C^{4k}),\ee
where $q=e^{2\pi i \tau}$,  be the Witten bundle, which is an element in $K(Z)[[q]]$. Construct the bundles 
\be \Theta(\widetilde{V})=\bigotimes_{u=1}^\infty
\Lambda_{-q^{u}}(V_\CC-\C^{2l}),\ \ \ \ \Theta_1(\widetilde{V})=\bigotimes_{u=1}^\infty
\Lambda_{q^u}(V_\CC-\C^{2l}),\ee
which are also elements in $K(Z)[[q]]$, and 
\be 
\Theta_2(\widetilde{V})=\bigotimes_{v=1}^\infty
\Lambda_{-q^{v-{1\over2}}}(V_\CC-\C^{2l}),\ \ \ \ \  \Theta_3(\widetilde{V})=\bigotimes_{v=1}^\infty \Lambda_{q^{v-{1\over2}}}(V_\CC-\C^{2l}),\ee
which are elements in $K(Z)[[q^{1/2}]]$. Let $\hat A(Z)$ be the $\hat A$-class of $TZ$ and $\Delta^\pm(V)$ the spinor bundles of $V$. The bundle twisted elliptic genera are defined to be the {\bf integral} $q$-series as follows,
\be \label{E} Ell(Z,V, \tau):=\int_Z \hat A(Z)\mathrm{Ch}(\Theta(\widetilde{T_\CC Z}))\mathrm{Ch}((\Delta^+(V)-\Delta^-(V))\otimes \Theta(\widetilde{V}))\in \Z[[q]],\ee
\be \label{E} Ell_1(Z,V, \tau):=\int_Z \hat A(Z)\mathrm{Ch}(\Theta(\widetilde{T_\CC Z}))\mathrm{Ch}((\Delta^+(V)+\Delta^-(V))\otimes \Theta_1(\widetilde{V}))\in \Z[[q]],\ee
\be \label{E2} Ell_2(Z,V, \tau):=\int_Z \hat A(Z)\mathrm{Ch}(\Theta(\widetilde{T_\CC Z}))\,\mathrm{Ch}(\Theta_2(\widetilde{V}))\in \Z[[q^{1/2}]],\ee
\be \label{E3} Ell_3(Z,V, \tau):=\int_Z \hat A(Z)\mathrm{Ch}(\Theta(\widetilde{T_\CC Z}))\,\mathrm{Ch}(\Theta_3(\widetilde{V}))\in \Z[[q^{1/2}]].\ee
Without the presence of $V$, 
\be W(Z)=\int_Z \hat A(Z)\mathrm{Ch}(\Theta(\widetilde{T_\CC Z}))\in \Z[[q]] \ee
is the famous Witten genus.
By the Atiyah-Singer index theorem, these bundle twisted elliptic genera have analytical interpretation as follows. Let ${\not\! \partial}^+$ be the spin Dirac operator on $Z$. Then 
\be \label{E} Ell(Z,V, \tau)={\rm Index}({\not\! \partial}^+\otimes\Theta(\widetilde{T_\CC Z})\otimes(\Delta^+(V)-\Delta^-(V))\otimes \Theta(\widetilde{V})))\in \Z[[q]],\ee
\be \label{E1} Ell_1(Z,V, \tau)={\rm Index}({\not\! \partial}^+\otimes\Theta(\widetilde{T_\CC Z})\otimes(\Delta^+(V)+\Delta^-(V))\otimes \Theta_1(\widetilde{V})))\in \Z[[q]],\ee
\be \label{E2} Ell_2(Z,V, \tau)={\rm Index}({\not\! \partial}^+\otimes\Theta(\widetilde{T_\CC Z})\otimes\Theta_2(\widetilde{V})))\in \Z[[q^{1/2}]],\ee
\be \label{E3} Ell_3(Z,V, \tau)={\rm Index}({\not\! \partial}^+\otimes\Theta(\widetilde{T_\CC Z})\otimes\Theta_3(\widetilde{V})))\in \Z[[q^{1/2}]],\ee
\be W(Z)={\rm Index}({\not\! \partial}^+\otimes\Theta(\widetilde{T_\CC Z})\in \Z[[q]]. \ee
One can show that when $p_1(Z)=p_1(V)$, these bundle twisted elliptic genera are modular forms of weight $2k$ over $SL(2, \Z), \Gamma_0(2), \Gamma^0(2)$ and $\Gamma_\theta$ respectively and when $p_1(Z)=0$, $W(Z)$ is a modular form of weight $2k$ over $SL(2, \Z)$ (see Appendix). 
Witten showed in \cite{W86, W99} that formally  ${\not\! \partial}^+\otimes\Theta(\widetilde{T_\CC Z})$ can be viewed as the Dirac operator on loop space; $(\Delta^+(V)-\Delta^-(V))\otimes \Theta(\widetilde{V}), (\Delta^+(V)+\Delta^-(V))\otimes \Theta_1(\widetilde{V}), \Theta_2(\widetilde{V})$, $\Theta_3(\widetilde{V})$ can be viewed as vector bundles over loop space; and
\begin{align*}
&{\not\! \partial}^+\otimes\Theta(\widetilde{T_\CC Z})\otimes(\Delta^+(V)-\Delta^-(V))\otimes \Theta(\widetilde{V}), \\
&{\not\! \partial}^+\otimes\Theta(\widetilde{T_\CC Z})\otimes(\Delta^+(V)+\Delta^-(V))\otimes \Theta_1(\widetilde{V}),\\
&{\not\! \partial}^+\otimes\Theta(\widetilde{T_\CC Z})\otimes\Theta_2(\widetilde{V}),\\
&{\not\! \partial}^+\otimes\Theta(\widetilde{T_\CC Z})\otimes\Theta_3(\widetilde{V})
\end{align*}
can be viewed as the Dirac operator on loop space coupled to these bundles. Taubes \cite{T89}, Bott-Taubes \cite{BT89} and Liu \cite{Liu96} proved the rigidity of these operators conjectured by Witten. In \cite{Liu95} Liu discovered a profound vanishing theorem for the Witten genus $W(Z)$. 

In this paper, we construct {\em projective elliptic genera} in the case that $Z$ is a compact oriented manifold {\em not necessarily spin} and $E$ is a {\em projective vector bundle} rather than an ordinary vector bundle on $Z$. More precisely, for such $Z$ and $E$, we construct {\bf rational} $q$-series
\be PEll(Z, E, \tau)\in \Q[[q]], \ PEll_1(Z, E, \tau)\in \Q[[q]],\ee 
\be PEll_2(Z, E, \tau)\in \Q[[q^{1/2}]], \  PEll_3(Z, E, \tau)\in \Q[[q^{1/2}]],\ee which still have both {\em  topological definition} and 
{\em analytic interpretation}. We also establish the modularity properties of these genera. The key new idea in the topological side is to introduce the {\em graded twisted Chern character} for Witten bundles constructed from projective vector bundles (see the definition in (\ref{GCh}) to (\ref{GCh3})). For the analytic interpretation, we use the projective spin Dirac operator introduced in \cite{MMS06, MMS08} and the fractional index theorem proved there. More precisely, let $\dirac^+$ be the projective spin Dirac operator associated to a fixed projective spin structure on the oriented compact manifold $Z$. Let $E$ be a projective complex vector bundle. The twisted projective spin Dirac operator $\dirac^+_E$ acts on $S\otimes E$, where $S$ is  the projective vector bundle of spinors associated to the Azumaya bundle given by the complex Clifford algebra bundle ${\text{Cliff}}(T_\CC M)$.  The index of $\dirac^+_E$ has the usual expression in terms of characteristic classes, it is no longer an integer, but only a fraction in general.

In 1983, the physicists Alvarez-Gaum\'e and Witten \cite{AGW} discovered the ``miraculous cancellation" formula for gravitational anomaly, relating index of the signature operator to indices of twisted Dirac operators in dimension 12. Liu \cite{Liu95cmp} generalised their formula to higher dimension and also allow general bundle twisting rather than the tangent bundle by developing modularities of certain characteristic forms . In this paper, we give a ``projective miraculous cancellation" formulae for indices of projective Dirac operators twisted by projective vector bundles (Theorem \ref{projmirac}) and the 12 dimensional local formula (Theorem \ref{projmirac12}) following Liu's method. 

The Witten genus can be viewed as a morphism from the String bordism ring to the ring of integral modular forms, $W:\Omega_{4r}^{String}\longrightarrow MF^\Z(SL(2, \Z)).$ It has a lift (called $\sigma-orientation$) \cite{AHS} in homotopy theory $\sigma: MString \longrightarrow tmf,$ where $tmf$ is the deep and powerful theory of topological modular forms, constructed originally by Hopkins and  Miller \cite{Hopkins}, with a new construction due to Lurie \cite{Lurie}. 
Our projective elliptic genus $PEll(Z, E, \tau)$ is a rational modular form over $SL(2, \Z)$ when the first rational Pontryagin classes of the projective bundle $E$ and $TZ$ are equal. It seems likely that there is a similar lift in homotopy theory for $PEll(Z, E, \tau)$ and a refinement of our projective genera to a version of elliptic cohomology, but we will not address this here. 

Other important approaches for construction of Witten genus and elliptic genera include chiral de Rham complex \cite{GMS00},\cite{GMS04}, \cite{BL00}, \cite{Ch12} and the application of factorization homology \cite{Cos10}. Our projective genera are twisted version of the usual genera in the presence of a $B$-field. We plan  to look at the construction of our projective genera in these approaches in the presence of a $B$-field. 

As an interesting application of the projective elliptic genera, we give a construction of the {\em elliptic pseudodifferential genera} for any elliptic pseudodifferential operator. More precisely, let $Z$ be a $4r$-dimensional compact oriented manifold. Choose and fix a projective spin$^c$ structure on $Z$ \cite{MMS06, MMS08}. Let $P$ be any elliptic pseudodifferential operator on $Z$. We are able to construct elliptic genera type invariants for $P$: $Ell(P, \tau)$ and $Ell_i(P, \tau), i=1,2,3$.  When $Z$ is a spin$^c$ manifold and $P$ is the spin$^c$ Dirac operator, $Ell(P, \tau)=0$ and $Ell_i(P, \tau)$ degenerate to the Witten genus of $Z$, which is a rational $q$-series on spin$^c$ manifold (see Example \ref{spinc}). This is similar to looking at the $\hat A$-genus on spin$^c$ manifolds in \cite{MMS06}.  The key step in our construction is to implicitly use a projective vector bundle coming from $P$ by using the projective spin$^c$ structure.  We also use the {\em Schur functors} (c.f. \cite{FH}) to understand the Witten bundles of tensor product.  Actually we give our construction in a more general setting, namely for projective elliptic pseudodifferential operator which has its own twist. 

The paper is organized as follows. In Section 1, we introduce the graded twisted Chern character on the Witten bundles constructed from a projective vector bundle and then construct the projective elliptic genera as well as study their modularities. In Section 2, we first review the index theory for projective elliptic operators in \cite{MMS06,MMS08} and then give the analytic interpretation of the projective elliptic genera.  We also give the ``projective miraculous cancellation" formula in this section. As an application, we construct projective elliptic pseudodifferential genera for projective elliptic pseudodifferential operator in Section 3.

\bigskip

\noindent{\bf Acknowledgements.} Fei Han was partially supported by the grant AcRF R-146-000-218-112 from National University of Singapore. He is indebted to Prof. Weiping Zhang for helpful suggestions and discussions. Varghese Mathai was supported by funding from the Australian Research Council, through the Australian Laureate Fellowship FL170100020. He gave a talk based on this paper at the conference, Microlocal methods in Analysis and Geometry (In honor of Richard Melrose’s 70th birthday) CIRM, Luminy, May 6-10 2019, and would like to thank Isadore Singer and Richard Melrose for past collaboration related to this research.


\section{Projective elliptic genera}

In this section, we give the topological construction of projective elliptic genera and study their modular properties.  We will give the the analytic interpretation of them in the next section by using the index theorem of projective elliptic operators in \cite{MMS06,MMS08}. 

\subsection{Projective vector bundles}
Let $Z$ be a smooth manifold with Riemannian metric and the Levi-Civita connection $\nabla^Z$. Let $Y$ be a principal $PU(N)$ bundle over $Z$,
$$
\begin{CD}
PU(N) @>>> \,  Y\\
&& @V \phi VV \\
&& Z \end{CD}
$$
The {\em  Dixmier-Douady invariant}
of $Y$,
 {}
 $$DD(Y) = \delta(Y) \in {\rm Torsion}(
H^3(Z, \mathbb Z))$$  
is the obstruction to lifting the principal $PU(N)$-bundle $Y$ to a principal $U(n)$-bundle
(the construction also works for any principal $G$ bundle $P$ over $Z$,
together with a central extension $\widehat G$ of $G$). Let $M_N(\mathbb C)$ be the algebra of $N\times N$ complex matrices. The  {associated algebra bundle} 
 $$
\cA = Y\times_{PU(N)} M_N(\mathbb C)
$$  
is called the associated {\em  Azumaya  bundle}.

A {\em projective vector bundle} on $Z$ is {\em not} a global bundle on $Z$, but rather it is a vector bundle $E\to Y$,
where $E$ also satisfies 
\be \label{module}
{\mathcal L}_g \otimes E_y \cong E_{g.y}, \qquad g\in PU(N), \; y\in Y,
\ee
where 
$\mathcal L= U(n)\times_{U(1)} \mathbb C\to PU(N)$ 
is the 
{primitive line bundle}, 
$$
{\mathcal L}_{g_1} \otimes {\mathcal L}_{g_2} \cong {\mathcal L}_{g_1. g_2}, \qquad g_i\in PU(N).
$$This gives a {projective action} of $PU(N)$ on $E$,
i.e. an action of $U(n)$ on $E$ s.t. the center $U(1)$ acts as scalars. One can define the {\em twisted Chern character} \cite{BCMMS}, $\mathrm{Ch}_{\delta{(Y)}}(E)\in H^{even}(Z, \Q)$ for the projective bundle $E$.

\subsection{Projective elliptic genera} Suppose $Z$ to be closed, oriented and $4r$-dimensional. Let $E$ be an Hermitian projective vector bundle of rank $l$ over $Z$, which is a Hermitian vector bundle over $Y$ with the action in $(\ref{module})$. Let $\nabla^E$ be an Hermitian connection on $E$ compatible with the action. Let $B\in \Omega^2(Y)$ be a curving of $Y$. Let $H\in \Omega^3(Z)$ representing the Dixmier-Douady class of $Y$ such that $\pi^*H=dB.$ As the Dixmier-Douady class is torsion element, $H$ is exact on $Z$. 

The condition (\ref{module}) implies that $\Tr(BI+R^E)^n$ descends to a degree $2n$ differential form on $Z$ for all $n\geq 0$. 

Define {\em  the first rational projective Pontryagin class of $E$}$, \mathfrak{p}_1(E)\in H^4(Z, \Q)$, such that 
\be \pi^*(\mathfrak{p}_1(E))=\frac{1}{4\pi^2}[\Tr(BI+R^E)^2]\in H^4(Y, \Q). \ee

The tensor product $E^{\otimes k}$ satisfies
\be 
{\mathcal L}_g^{\otimes k} \otimes E_y^{\otimes k}  \cong E_{g.y}^{\otimes k} , \qquad g\in PU(N), \; y\in Y.
\ee Therefore we see that $\Tr\left(kBI+R^{E^{\otimes k}}\right)^n$ descends to a degree $2n$ differential form on $Z$ for all $n\geq 0$. 
We can define the twisted Chern character 
\be \mathrm{Ch}_{kH}(E^{\otimes k})=\Tr\left(\exp\left[\frac{i}{2\pi}\left(kBI+R^{E^{\otimes k}}\right)\right]\right).   \ee
As the exterior bundle $\wedge^k E$ is a subbundle of $E^{\otimes k}$, one can also define the twisted Chern character $\mathrm{Ch}_{kH}(\wedge^k E)$.

Recall that for an indeterminate $t$ (c.f. \cite{A67}), 
\be \Lambda_t(E)=\CC
|_M+tE+t^2\wedge^2(E)+\cdots,\ \ \ S_t(E)=\CC |_M+tE+t^2
S^2(E)+\cdots, \ee are the total exterior and
symmetric powers of $E$ respectively. 

Let \cite{W86}
\be \Theta(T_\CC Z)=\bigotimes_{n=1}^\infty
S_{q^n}(T_\CC Z).\ee
be the Witten bundle, which is an element in $K(Z)[[q]]$.

Let $\bar E$ be the complex conjugate of $E$, which carries the induced Hermitian metric and connection. 

Set (c.f.\cite{Liu96, Liu95})
\be \Theta(E)=\bigotimes_{u=1}^\infty
\Lambda_{-q^{u}}(E)\otimes \bigotimes_{u=1}^\infty
\Lambda_{-q^{u}}(\bar E),\ \ \  \Theta_1(E)=\bigotimes_{u=1}^\infty
\Lambda_{q^u}(E)\otimes \bigotimes_{u=1}^\infty
\Lambda_{q^u}(\bar E),\ee
which are elements in $K(Y)[[q]]$;
\be \Theta_2(E)=\bigotimes_{v=1}^\infty
\Lambda_{-q^{v-{1\over2}}}(E)\otimes \bigotimes_{v=1}^\infty
\Lambda_{-q^{v-{1\over2}}}(\bar E),\ \ \ \Theta_3(E)=\bigotimes_{v=1}^\infty
\Lambda_{q^{v-{1\over2}}}(E)\otimes \bigotimes_{v=1}^\infty
\Lambda_{q^{v-{1\over2}}}(\bar E),\ee
which are elements in $K(Y)[[q^{1/2}]]$.

In the $q$-expansion of $(\wedge^{even}E-\wedge^{odd}E)\otimes\Theta(E)$, the coefficient of $q^{n}$ is integral linear combination of terms of the form $$\wedge^{i_1}(E)\otimes\wedge^{i_2}(E)\otimes \cdots \wedge^{i_k}(E)\otimes \wedge^{j_1}(\bar E)\otimes \wedge^{j_2}(\bar E)\otimes \cdots \wedge^{j_l}(\bar E). $$ Pick out the the terms such that $(i_1+i_2+\cdots i_k)-(j_1+j_2+\cdots j_l)=m$ and denote their sum by $W_{m, n}(E)$. Then we have the expansion
\be (\wedge^{even}E-\wedge^{odd}E)\otimes\Theta(E)=\sum_{n=0}^\infty \left(\sum_{m=-\infty}^\infty W_{m, n}(E)\right)q^{n}.   \ee 
$W_{m,n}(E)$ is a vector bundle over $Z$ carrying induced Hermitian metric and connection for each $m$. It is clear that for each fixed $n$, there are only finite many $m$ such that $W_{m, n}(E)$ is nonzero. $W_{m,n}(E)$ satisfies
\be 
{\mathcal L}_g^{\otimes m} \otimes W_{m,n}(E)_y  \cong W_{m,n}(E)_{g.y}, \qquad g\in PU(N), \; y\in Y.
\ee Therefore one can define the twisted Chern character $\mathrm{Ch}_{mH}(W_{m,n}(E))$. 

Similarly in the $q$-expansion of $(\wedge^{even}E+\wedge^{odd}E)\otimes\Theta_1(E)$, the coefficient of $q^{n}$ is integral linear combination of terms of of the form $$\wedge^{i_1}(E)\otimes\wedge^{i_2}(E)\otimes \cdots \wedge^{i_k}(E)\otimes \wedge^{j_1}(\bar E)\otimes \wedge^{j_2}(\bar E)\otimes \cdots \wedge^{j_l}(\bar E). $$ Pick out the the terms such that $(i_1+i_2+\cdots i_k)-(j_1+j_2+\cdots j_l)=m$ and denote their sum by $A_{m, n}(E)$.  Then we have the expansion
\be  (\wedge^{even}E+\wedge^{odd}E)\otimes\Theta_1(E)=\sum_{n=0}^\infty \left(\sum_{m=-\infty}^\infty A_{m, n}(E)\right)q^{n}.   \ee 
$A_{m,n}(E)$ is a vector bundle over $Z$ carrying induced Hermitian metric and connection for each $m$. For each fixed $n$, there are only finite many $m$ such that $A_{m, n}(E)$ is nonzero. $A_{m,n}(E)$ satisfies
\be 
{\mathcal L}_g^{\otimes m} \otimes A_{m,n}(E)_y  \cong A_{m,n}(E)_{g.y}, \qquad g\in PU(N), \; y\in Y.
\ee We can therefore define the twisted Chern character $\mathrm{Ch}_{mH}(A_{m,n}(E))$.

One can decompose $\Theta_2(E)$ and $\Theta_3(E)$ in a similar way as
\be \Theta_2(E)=\sum_{n=0}^\infty \left(\sum_{m=-\infty}^\infty B_{m, n}(E)\right)q^{n/2},  \ee
\be \Theta_3(E)=\sum_{n=0}^\infty \left(\sum_{m=-\infty}^\infty C_{m, n}(E)\right)q^{n/2},  \ee
and define the twisted Chern characters $\mathrm{Ch}_{mH}(B_{m,n}(E))$ and  $\mathrm{Ch}_{mH}(C_{m,n}(E))$ as well. 

Define the {\em graded twisted Chern character}
\be \label{GCh}\GCh_H\left((\wedge^{even}E-\wedge^{odd}E)\otimes\Theta(E)\right)=\sum_{n=0}^\infty\left(\sum_{m=-\infty}^\infty\mathrm{Ch}_{mH}(W_{m, n}(E))\right)q^n\in \Omega^*(M)[[q]],\ee
\be \label{GCh1} \GCh_H\left((\wedge^{even}E+\wedge^{odd}E)\otimes\Theta_1(E)\right)=\sum_{n=0}^\infty\left(\sum_{m=-\infty}^\infty\mathrm{Ch}_{mH}(A_{m, n}(E))\right)q^n\in \Omega^*(M)[[q]],\ee
\be \label{GCh2} \GCh_H(\Theta_2(E))=\sum_{n=0}^\infty\left(\sum_{m=-\infty}^\infty\mathrm{Ch}_{mH}(B_{m, n}(E))\right)q^{n/2}\in \Omega^*(M)[[q^{1/2}]],\ee
\be \label{GCh3} \GCh_H(\Theta_3(E))=\sum_{n=0}^\infty\left(\sum_{m=-\infty}^\infty\mathrm{Ch}_{mH}(C_{m, n}(E)\right)q^{n/2}\in \Omega^*(M)[[q^{1/2}]].\ee

Let (compare with (\ref{reducedwittenbundle}))
\be \Theta(T_\CC Z)=\bigotimes_{n=1}^\infty
S_{q^n}(T_\CC Z). \ee

Let $\det \bar E$ be the determinant line bundle of $\bar E$, which carries the induced Hermitian metric and connection. As ther determinant is the highest exterior power, one can define the twisted Chern character $\mathrm{Ch}_{-lH}(\det \bar E)$. 

Define the {\em projective elliptic genera} by
\be \label{P1}
\begin{split}
& PEll(Z,E, \tau)\\
:=& \left(\prod_{j=1}^\infty(1-q^j)\right)^{4r-2l}\cdot\int_Z \hat A(Z)\mathrm{Ch}(\Theta(T_\CC Z))\sqrt{\mathrm{Ch}_{-lH}(\det \bar E)}\, \GCh_H\left((\wedge^{even}E-\wedge^{odd}E)\otimes\Theta(E)\right)\\
&\in \Q[[q]],
\end{split}
\ee
\be \label{P2}  
\begin{split}
&PEll_1(Z,E, \tau)\\
:=&\frac{\left(\prod_{j=1}^\infty(1-q^j)\right)^{4r}}{\left(\prod_{j=1}^\infty(1+q^j)\right)^{2l}}\cdot\int_Z \hat A(Z)\mathrm{Ch}(\Theta(T_\CC Z))\sqrt{\mathrm{Ch}_{-lH}(\det \bar E)}\,\GCh_H\left((\wedge^{even}E+\wedge^{odd}E)\otimes\Theta_1(E)\right)\\
&\in \Q[[q]],
\end{split}
\ee
\be \label{P3} PEll_2(Z,E, \tau):=\frac{\left(\prod_{j=1}^\infty(1-q^j)\right)^{4r}}{\left(\prod_{j=1}^\infty(1-q^{j-1/2})\right)^{2l}}\cdot\int_Z \hat A(Z)\mathrm{Ch}(\Theta(T_\CC Z))\,\GCh_H(\Theta_2(E))\in \Q[[q^{1/2}]],\ee
\be \label{P4} PEll_3(Z,E, \tau):=\frac{\left(\prod_{j=1}^\infty(1-q^j)\right)^{4r}}{\left(\prod_{j=1}^\infty(1+q^{j-1/2})\right)^{2l}}\cdot\int_Z \hat A(Z)\mathrm{Ch}(\Theta(T_\CC Z))\,\GCh_H(\Theta_3(E))\in \Q[[q^{1/2}]].\ee

We have the following results. 
\begin{theorem}\label{modularity} (i) If $p_1(TZ)=\mathfrak{p}_1(E)$, then $PEll(Z,E, \tau)$ is a modular form of weight $2r$ over $SL(2, \Z)$. \newline
(ii) If $p_1(TZ)=\mathfrak{p}_1(E)$, then $PEll_1(Z,E, \tau)$ is a modular form of weight $2r$ over $\Gamma_0(2)$, $PEll_2(Z,E, \tau)$ is a modular form of weight $2r$ over $\Gamma^0(2)$ and $PEll_3(Z,E, \tau)$ is a modular form of weight $2r$ over $\Gamma_\theta(2)$; moreover, we have 
\be PEll_1(Z,E, -1/\tau)=\tau^{2r}PEll_2(Z,E, \tau), \ \ PEll_2(Z,E, \tau+1)=PEll_3(Z,E, \tau).\ee 
\end{theorem}

\begin{proof} Let $g^{TZ}$ be a Riemann metric on $TZ$, $\nabla^{TZ}$ be the Levi-Civita connection and $R^{TZ}=\left(\nabla^{TZ}\right)^2$ be the curvature of $\nabla^{TZ}$.  The $\hat A$-form can be expressed as
$$\hat{A}(Z, \nabla^{TZ})={\det}^{1/2}\left({{\sqrt{-1}\over
4\pi}R^{TZ} \over \sinh\left({ \sqrt{-1}\over
4\pi}R^{TZ}\right)}\right)$$
By the Chern-Weil theory (c.f. \cite{Z01}), the Chern-Weil expression of twisted Chern characters (\cite{BCMMS}) and the definitions of Jacobi theta functions and the Jacobi identity (see Appendix), we have
\be PEll(Z,E, \tau)=\int_Z \mathrm{det}^{1\over
2}\left(\frac{R^{TZ}}{4{\pi}^2}\frac{\theta'(0,\tau)}{\theta(\frac{R^{TZ}}{4{\pi}^2},\tau)}\right)
 \mathrm{det}\left(\frac{4\pi^2}{BI+R^{E}}\frac{\theta(\frac{BI+R^{E}}{4{\pi}^2},\tau)}{\theta'(0,\tau)}\right), \ee

\be PEll_1(Z,E, \tau)=\int_Z \mathrm{det}^{1\over
2}\left(\frac{R^{TZ}}{4{\pi}^2}\frac{\theta'(0,\tau)}{\theta(\frac{R^{TZ}}{4{\pi}^2},\tau)}\right)
 \mathrm{det}\left(\frac{\theta_{1}(\frac{BI+R^{E}}{4{\pi}^2},\tau)}{\theta_{1}(0,\tau)}\right), \ee

\be PEll_2(Z,E, \tau)=\int_Z \mathrm{det}^{1\over
2}\left(\frac{R^{TZ}}{4{\pi}^2}\frac{\theta'(0,\tau)}{\theta(\frac{R^{TZ}}{4{\pi}^2},\tau)}\right)
 \mathrm{det}\left(\frac{\theta_{2}(\frac{BI+R^{E}}{4{\pi}^2},\tau)}{\theta_{2}(0,\tau)}\right),  \ee

\be PEll_3(Z, E, \tau)=\int_Z \mathrm{det}^{1\over
2}\left(\frac{R^{TZ}}{4{\pi}^2}\frac{\theta'(0,\tau)}{\theta(\frac{R^{TZ}}{4{\pi}^2},\tau)}\right)
 \mathrm{det}\left(\frac{\theta_{3}(\frac{BI+R^{E}}{4{\pi}^2},\tau)}{\theta_{3}(0,\tau)}\right).  \ee

It is known that the generators of $\Gamma_0(2)$ are $T,ST^2ST$, the generators
of $\Gamma^0(2)$ are $STS,T^2STS$  and the generators of
$\Gamma_\theta$ are $S$, $T^2$.  Applying the Chern root algorithm on the level of forms (over certain ring extension $\CC[\wedge^2_xT^*M]\subset R'$ for each $x\in M$, cf. \cite{Hu} for details) and the transformation laws of the theta functions, one can see that that when $\Tr (R^{TZ})^2$ and $\Tr (BI+R^E)^2$ are cohomologous, the anomalies arising from modular transformation $S: \tau\to -1/\tau$ vanish. 

\end{proof}


\section{Analytic definition of projective elliptic genera}
In this section, we first briefly review the fractional index theorem of Mathai-Melrose-Singer \cite{MMS06,MMS08} and then use it to give an analytic construction of the projective elliptic genera. 

\subsection{Projective elliptic operators}
For a compact manifold, $Z,$ and vector bundles $E$ and $F$ over $Z$,
the {\em  Schwartz kernel theorem}
gives a 1-1 correspondence,
$$ 
 \text{  continuous\, linear \,operators},  \quad T:C^\infty(Z, E) \longrightarrow C^{-\infty}(Z,F)
 $$
 $$
{ \Bigg\Updownarrow}\qquad\qquad\qquad { \Bigg\Updownarrow}
 $$
{$$
 \text{  distributional\, sections}, 
 \quad k_T\in C^{-\infty}(Z^2,\Hom(E, F)\otimes \Omega_R)
$$}{}
where $\Hom(E, F)_{(z,z')} = F_z\boxtimes E_{z'}^*$
is the `big' homomorphism bundle over $Z^2$ and $\Omega_R$ is the density
bundle from the right factor. 
  
When restricted to {pseudodifferential
operators}, $\quad \Psi^m(Z, E, F)$, 
get an isomorphism with the space
of {\em  conormal distributions} with
respect to the diagonal,
$I^m(Z^2,\Delta;\Hom(E,F)).$ i.e.
$$\Psi^m(Z, E, F)\quad \Longleftrightarrow\quad I^m(Z^2,\Delta;\Hom(E,F))$$
When further restricted to differential
operators $\,{\rm Diff}^m(Z, E, F)$ (which by definition have the property of being local operators) 
this becomes an isomorphism with the space of conormal distributions,
$\quad I^m_{\Delta} (Z^2,\Delta;\Hom(E,F)),\quad$ 
 with respect to the diagonal,     {\em  supported}      {\em  within the diagonal},  $\Delta$. i.e.
$${\rm Diff}^m(Z, E, F)\quad \Longleftrightarrow \quad I^m_{\Delta} (Z^2,\Delta;\Hom(E,F))$$

The previous facts  motivates our definition of projective differential and pseudodifferential 
operators when $E$ and $F$ are only projective vector bundles associated to
a fixed finite-dimensional Azumaya bundle $\cA.$

Since a projective vector bundle $E$ is not global on $Z$, one
   {   }{\em  cannot} make sense of sections of $E$, let alone 
operators acting between sections!
However, it still makes sense to talk about 
Schwartz kernels even in this case, as we explain.

Notice that $\Hom(E,F) = F \boxtimes E^*$ is a projective bundle on $Z^2$ associated to 
the Azumaya bundle,
$\cA_L\boxtimes\cA_R'$.

The restriction $\Delta^*\Hom(E,F) = \hom(E,F) $ to the diagonal 
is an ordinary vector bundle, 
it is therefore reasonable to expect that $\Hom(E,F)$ also restricts to an ordinary vector bundle
in a tubular nbd  $N_\epsilon$ of the diagonal.

In    \cite{MMS06}, it is shown that there is a canonical such choice, $\Hom^{\cA}(E,F)$, 
such that the {composition properties hold}.

This allows us to    {   }{define} the space of     {   }{projective 
pseudo- differential operators} { $\Psi^\bullet_\epsilon(Z; E, F)$} with Schwartz kernels supported in an
$\epsilon$-neighborhood $N_\epsilon$ of the diagonal $\Delta$ in  $Z^2$, with the space of 
{   {   }{conormal distributions}},
{$I^\bullet_\epsilon (N_\epsilon, \Delta ; \Hom^{\cA}(E, F)).$}
$$\Psi^\bullet_\epsilon(Z; E, F)\quad :=\quad I^\bullet_\epsilon (N_\epsilon,\Delta ;\Hom^{\cA}(E, F)).
$$

Despite    {   }{\em  not} being a space of operators, this has precisely the    {   }{
same local
structure} as in the standard case and has similar composition properties
provided supports are restricted to appropriate neighbourhoods of the
diagonal.

The space of    {   }{\em  projective }      {   }{\em  smoothing operators},
$\Psi^{-\infty}_\epsilon(Z; E, F)$ 
is defined as the smooth sections, $C^\infty_c(N_\epsilon;\Hom^{\cA}(E, F)
\otimes \pi_R^*\Omega).$

The  space of all     {   }{projective differential operators},
{ ${\rm Diff}^\bullet(Z; E, F)$} is defined as those
conormal distributions that are    {   }{supported within
the diagonal} $\Delta$ in ${Z^2}$,

$${\rm Diff}^\bullet(Z; E, F):=I^\bullet_{\Delta} (N_\epsilon,\Delta;\Hom^{\cA}(E, F)).$$

In fact, 
 ${\rm Diff}^\bullet(Z; E, F)$ is even a      {   }{\em  ring} when $E=F$.

Recall that there is a projective
bundle of spinors  $\cS = \cS^+\oplus \cS^-$
on any even dimensional oriented manifold $Z$.

There are natural spin connections on the Clifford algebra bundle
$Cl(Z)$ and $\cS^\pm$ induced from the
Levi-Civita connection on $T^*Z.$

Recall also that 
$\hom(\cS,\cS) \cong Cl(Z),$ has
an extension to $\tilde Cl(Z)$ in a tubular neighbourhood of the diagonal $\Delta$,  with 
an induced connection $\nabla$.

The    {   }{\em  projective spin Dirac operator} is defined as the distributional section
$$
{\not\! \partial}=cl\cdot\nabla_{L}(\kappa_{Id}),\qquad \kappa_{Id}=\delta (z-z') Id_S
$$

Here
$\nabla_{L}$ is the connection $\nabla$ restricted to the left variables with
$cl$ the contraction given by the Clifford action of $T^*Z$ on the
left.

As in the usual case, the  projective 
spin Dirac operator ${\not\! \partial}$ is    {   }{\em  elliptic} 
and odd with respect to  $\bbZ_2$ grading 
of $\cS$.

The    {   }{\em  principal symbol map} is well defined for
conormal distributions, leading to the globally defined symbol map,
 {}
$$\sigma: \Psi^m_\epsilon(Z; E, F) \longrightarrow C^\infty(T^*Z, \pi^*\hom(E,F)),$$
homogeneous of
degree $m$; here  $\hom(E,F),$ is a  globally defined {\em  ordinary vector bundle} with fibre 
$\hom(E,F)_z = F_z \otimes E_z^*.$
Thus    {   }{\em  ellipticity} is well defined, as the
   {   }{\em  invertibility of this symbol}.

{ Equivalently}, $A \in \Psi^m_{\epsilon/2}(Z; E, F)$ is    {   }{elliptic}
if there exists a parametrix $B \in \Psi^{-m}_{\epsilon/2}(Z; F, E)$ and  smoothing operators
$\; Q_R \in \Psi^{-\infty}_{\epsilon}(Z; E, E),\;$  $Q_L\in \Psi^{-\infty}_{\epsilon}(Z; F, F)\;$ 
such that
 {}
$$
BA = I_E - Q_R, \qquad\qquad AB = I_F - Q_L
$$

The    {   }{\em  trace functional} is defined on projective smoothing operators\,\,
$\;
\Tr:\Psi^{-\infty}_{\epsilon}(Z;E) \to \mathbb C
\;$
as
 {}
$$
\Tr(Q)=\int_Z \tr Q(z,z).
$$
It vanishes on commutators, i.e.
$
\quad \Tr(QR-RQ)=0, \quad
$
if 
$$\quad Q\in\Psi^{-\infty}_{\epsilon/2}(Z;F,E), R\in\Psi^{-\infty}_{\epsilon/2}(Z;E,F),$$
which follows from Fubini's theorem.

The  {\em  fractional analytic index} of the projective elliptic operator
$A \in \Psi^\bullet_\epsilon(Z; E, F)$ is defined in the essentially analytic way as,
$$
{\rm Index}_a(A)=\Tr([A,B]) \in \mathbb R,
$$
where $B$ is a parametrix for $A,$ and the 
RHS is the notation for 
$
\Tr_F(AB- I_F) - \Tr_E(BA-I_E).
\quad$

For $A\in \Psi^m_{\epsilon /4}(Z;E,F)$, the Guillemin-Wodzicki
{\em  residue trace} is,
$$
\Tr_R(A)=\lim_{z\to0}z\Tr(AD(z)),
$$
where 
$D(z)\in\Psi^z_{\epsilon/4}(Z;E)$ is an    {   }{ entire family}
of $\Psi$DOs of complex order $z$ which is
elliptic and such that $D(0)=I.$ The residue trace is independent of the
choice of such a family. 

 \begin{enumerate}
\item The residue trace $\Tr_R$ vanishes on all $\Psi$DOs of
sufficiently negative order. 
\item  The residue trace $\Tr_R$ is also a trace functional, that is, 
$$
\Tr_R([A,B])=0,
$$
for $A\in \Psi^m_{\epsilon /4}(Z;E,F),\
B\in\Psi^{m'}_{\epsilon/4}(Z;F,E).$
\end{enumerate}

The {\em  regularized trace}, is defined to be the residue,
$$
\Tr_D(A)=\lim_{z\to0} \frac1z\left(z\Tr(AD(z))-\Tr_R(A)\right).
$$ For general $A$, $\Tr_D$  does depend on the regularizing family $D(z)$. But for
smoothing operators it coincides with the standard operator trace, 
$$
\Tr_D(S) = \Tr(S), \qquad \forall S\in \Psi^{-\infty}_\epsilon(Z, E).
$$

Therefore {\em the fractional analytic index} is also given by,
$$
{\rm Index}_a(A) = \Tr_D([A,B]),
$$ for a projective elliptic operator $A$, and $B$ a parametrix for $A$.

The regularized trace $\Tr_D$ is    {   }{\em  not} a trace function, but
however it satisfies the {\em  trace defect formula},
$$\Tr_D([A,B])=\Tr_R(B\delta _DA),$$
where $\delta _D$ is a {\em  derivation} acting on the full symbol algebra. It also satisfies the condition of being closed,
$$
\Tr_R(\delta _Da)=0 \quad \forall  \; a.
$$
  
Using the derivation  $\delta _D$ and the trace defect formula, 
we prove:
\begin{enumerate}
\item the {\em   homotopy invariance} of the index,
$$\frac{d}{dt}{\rm Index}_a(A_t) = 0,$$ where $\quad t\mapsto A_t \quad$ is a smooth 1-parameter family of 
projective elliptic $\Psi$DOs;
 
\item the {\em  multiplicativity property} of 
of the index, $${\rm Index}_a(A_2A_1) = {\rm Index}_a(A_1)+ {\rm Index}_a(A_2),$$ where $A_i$ for $i=1,2$ are projective elliptic $\Psi$DOs.
\end{enumerate}

An analogue of the {\em  McKean-Singer formula} holds,
$${\rm Index}_a({\not\! \partial}_E^+)=\lim_{t\downarrow0}\Tr_s(H_\chi(t)),$$
where $H_\chi(t) = \chi(H_t)$ is a globally defined,  {\em  truncated
heat kernel}, both in space (in a neighbourhood of the diagonal) and in time. The local index theorem can then be applied, thanks to the McKean-Singer formula,
to obtain the {\em  index theorem} for projective spin Dirac operators. 

\begin{theorem}[\cite{MMS06}]\label{MMS}
{The  projective 
spin Dirac operator 
 on an
even-dimensional compact oriented manifold $Z$,
has fractional analytic index,
 $$
{\rm Index}_a({\not\! \partial}^+) = \int_Z \widehat A(Z)  \in \mathbb Q.
$$}
\end{theorem}

Recall that $Z={\mathbb C}P^{2n}$ is an oriented but  {\em  non-spin} (actually spin$^c$) manifold such that  
$\displaystyle\int_Z{\widehat A} (Z) \not\in \mathbb Z$, e.g.

 {}
$$
\begin{aligned}
Z ={\mathbb C} P^2\qquad & \Longrightarrow\qquad {\rm Index}_a({\not\! \partial}^+)=-1/8.\\
Z = {\mathbb C} P^{4} \qquad &\Longrightarrow\qquad {\rm Index}_a({\not\! \partial}^+) = 3/128.
\end{aligned}
$$

\subsection{Transversally elliptic operators and projective elliptic operators}
Recall that projective half spinor bundles $\cS^\pm$ on $Z$ can be realized as
$Spin(n)$-equivariant honest vector bundles, $(\hat \cS^+,\hat \cS^-),$
over the total space of the oriented frame bundle $\cP$ 
and in which the center, $\bbZ_2,$ acts
as $\pm 1$, as follows:\\   the conormal bundle $N$ to the fibres of $\cP$
has vanishing $w_2$-obstruction, and $\hat \cS^\pm$ are just the 1/2 spin 
bundles of $N$.

One can define the $Spin(n)$-equivariant 
transversally elliptic Dirac operator  $\hat{\not\!\partial}^\pm$  
using the Levi-Civita connection on  $Z$ together with the 
Clifford contraction, where    {   }{\em  transverse ellipticity} means 
that the principal symbol is invertible when restricted  to directions that are conormal to the fibres.

The nullspaces of  $\hat{\not\!\partial}^\pm$ are 
infinite dimensional unitary representations of $Spin(n)$.
The transverse ellipticity implies that the characters of these representations
are    {   }{\em  distributions} on the group $Spin(n)$. In particular, the 
   {   }{\em  multiplicity} of each irreducible 
unitary representation in these nullspaces is    {   }{\em  finite},
and grows at most polynomially.  

The    {   }{\em  $Spin(n)$-equivariant index} of $\hat{\not\!\partial}^+$ is defined to be the following
distribution on  $Spin(n)$, 
$$
{\rm Index}_{Spin(n)}(\hat{\not\!\partial}^+) = {\rm Char}({\rm Nullspace}((\hat{\not\!\partial}^+))
- {\rm Char}({\rm Nullspace}(\hat{\not\!\partial}^-))
$$
An alternate,   {\em  analytic} description of the  $Spin(n)$-equivariant index of $\hat{\not\!\partial}^+$
is: for a function of compact support $\chi\in\CIc(G),$ 
the action of the group induces a graded operator
$$
T_{\chi}:C^\infty(\cP;\hat\cS)\longrightarrow C^\infty(\cP;\hat\cS),\quad
T_{\chi}u(x)=\int_{G}\chi(g)g^*udg,
$$
which is smoothing along the fibres.  
 $\hat{\not\!\partial}^+$ has a microlocal parametrix $Q,$  in the 
 directions that are conormal to the fibres (ie along $N$).
Then for any $\chi\in\CIc(G),$ 
$$
T_{\chi}\circ ( \hat{\not\!\partial}^+\circ Q-I_-)\in\Psi^{-\infty}(\cP;\hat\cS^-);
\quad
T_{\chi}\circ (Q \circ \hat{\not\!\partial}^+-I_+)\in\Psi^{-\infty}(\cP;\hat\cS^+)
$$
are smoothing operators. The     {   }{\em  $Spin(n)$-equivariant 
index} of $\hat{\not\!\partial}^+$,
evaluated at $\chi\in\CIc(G)$, is also given by:  
$$
{\rm Index}_{Spin(n)}(\hat{\not\!\partial}^+)(\chi)=\Tr(T_{\chi}\circ ( \hat{\not\!\partial}^+\circ Q-I_-))-
\Tr(T_{\chi}\circ (Q \circ \hat{\not\!\partial}^+-I_+))
$$
 \begin{theorem}[\cite{MMS08}]
Let  $\pi: \cP^2 \to Z^2$ denote the projection.
The pushforward map, $\pi_*$, maps the Schwartz kernel of the $Spin(n)$-transversally elliptic
Dirac operator to the projective Dirac operator: 
That is, 
$$\pi_*(\hat{\not\!\partial}^\pm) = {\not\!\partial}^\pm.$$
\end{theorem}

We will next relate these two pictures. An easy argument shows that the support of the equivariant index distribution is contained 
within the center $\bbZ_2$ of $Spin(n)$.

\begin{theorem}[\cite{MMS08}]\label{thm:equiv}
Let $\phi\in\CI(Spin(n))$ be such that :
\begin{enumerate}
\item $\phi \equiv 1$ in a neighborhood of $e$, the identity of $Spin(n)$;
\item $-e \not\in {\rm supp}(\phi)$. Then
\end{enumerate}
$$
{\rm Index}_{Spin(n)}(\hat\dirac^+\otimes \hat E)(\phi) = {\rm Index}_a(\dirac^+\otimes E)
$$
where $E$ is a projective vector bundle associated to $\cP$ on $Z$ and $\hat E$ is the lift of $E$ to $\cP$.
\end{theorem}
 
Informally, the fractional analytic index,
of the projective Dirac operator $\dirac^+$, 
is the coefficient of the delta function (distribution) at 
 the identity in $Spin(n)$ of the $Spin(n)$-equivariant index for the associated
transversally elliptic Dirac operator $\hat{\dirac}^+$ on $\cP.$ 

\subsection{Analytic definition of projective elliptic genera}
In view of the Definition \ref{P1} to \ref{P4}, Theorem \ref{MMS} and Theorem \ref{thm:equiv}, we obtain the following analytic expressions  
\begin{align*}
PEll_2(Z,E,\tau)=& \ind_a(\dirac^+\otimes(\Theta(T_\CC Z)\otimes \Theta_2(E))\\
=&  \ind_{Spin(n)}(\hat\dirac^+ \otimes (\widehat{\Theta(T_\CC Z)\otimes \Theta_2(E)})(\phi)  \in \QQ[[q^{1/2}]];\\
PEll_3(Z,E,\tau) =& \ind_a(\dirac^+\otimes(\Theta(T_\CC Z)\otimes \Theta_3(E))\\
=&  \ind_{Spin(n)}(\hat\dirac^+ \otimes (\widehat{\Theta(T_\CC Z)\otimes \Theta_3(E)})(\phi)  \in \QQ[[q^{1/2}]],
\end{align*}
where $\phi$ is as in Theorem \ref{thm:equiv}.

When $c_1(E)$ is even and $l$ is even, the square root line bundle $\sqrt{\det \bar E}$ exists and satisfies
\be {\mathcal L}_g^{\otimes l/2} \otimes \sqrt{\det \bar E}_y  \cong \sqrt{\det \bar E}_{g.y}, \qquad g\in PU(N), \; y\in Y.
\ee
Then we further have
\be
\begin{split}
&PEll(Z,E,\tau)
\\ =& \ind_a(\dirac^+\otimes(\Theta(T_\CC Z)\otimes \sqrt{\det \bar E}\otimes (\wedge^{even}E-\wedge^{odd}E)\otimes\Theta(E))\\
=&  \ind_{Spin(n)}(\hat\dirac^+ \otimes (\widehat{\Theta(T_\CC Z)\otimes \sqrt{\det \bar E}\otimes (\wedge^{even}E-\wedge^{odd}E)\otimes\Theta(E)})(\phi)  \in \QQ[[q]];
\end{split}
\ee
\be
\begin{split}
&PEll_1(Z,E,\tau)
\\ =& \ind_a(\dirac^+\otimes(\Theta(T_\CC Z)\otimes \sqrt{\det \bar E}\otimes (\wedge^{even}E+\wedge^{odd}E)\otimes\Theta_1(E))\\
=&  \ind_{Spin(n)}(\hat\dirac^+ \otimes (\widehat{\Theta(T_\CC Z)\otimes \sqrt{\det \bar E}\otimes (\wedge^{even}E+\wedge^{odd}E)\otimes\Theta_1(E)})(\phi)  \in \QQ[[q]],
\end{split}
\ee
where $\phi$ is as in Theorem \ref{thm:equiv}.

\subsection{Projective miraculous cancellation formula}

We have the following ``projective miraculous cancellation formula" for projective Dirac operators, generalizing the celebrated Alvarez-Gaum\'e-Witten ``miraculous cancellation formula" \cite{AGW}  in dimension 12 for ordinary Dirac operators.
 
Let
\begin{align*}
\Theta(T_\CC Z)\otimes \Theta_2(E)
=&\bigotimes_{n=1}^\infty
S_{q^n}(T_\CC Z)\otimes\bigotimes_{v=1}^\infty
\Lambda_{-q^{v-{1\over2}}}(E)\otimes \bigotimes_{v=1}^\infty
\Lambda_{-q^{v-{1\over2}}}(\bar E)\\
=& \sum_{i=0}^\infty B_i(T_\CC Z, E)q^{i/2},
\end{align*}
where each $B_i(T_\CC Z, E)$ is a virtual projective bundle over $Z$. 
 
The following {\em projective miraculous cancellation} can be similarly proved as Theorem 1 in \cite{Liu95cmp} by using the modularities in Theorem \ref{modularity} and the looking at the basis of rings of modular forms over $\Gamma_0(2)$ and $\Gamma^0(2)$. 

\begin{theorem}\label{projmirac} If $p_1(TZ)=\mathfrak{p}_1(E)$, $c_1(E)$ is even and $l$ is even, then the following equality holds,
\be\label{eq:miraculous} 
\ind_a(\dirac^+\otimes \sqrt{\det \bar E}\otimes (\wedge^{even}E+\wedge^{odd}E))=\sum_{j=0}^{[\frac{k}{2}]}2^{l+k-6[\frac{k}{2}]} \ind_a(\dirac^+\otimes h_j(T_\CC Z, E)),  
\ee
where the virtual projective bundles $h_j(T_\CC Z, E), 0\leq j\leq [\frac{k}{2}]$ are canonical integral linear combination of $B_i(T_\CC M, E), 0\leq i\leq j$. 
\end{theorem}

When the dimension of $Z$ is 12, the local formula of the projective miraculous cancellation formula for the top (degree 12) forms reads
\begin{theorem} \label{projmirac12}
\be  
\begin{split}
&\{\hat A(TZ, \nabla^{TZ})\Ch_{\left(2^{l-1}-\frac{l}{2}\right) H}(\sqrt{\det \bar E}\otimes (\wedge^{even}E+\wedge^{odd}E))\}^{(12)}\\
=&2^{l-3}\left( \{\hat A(TZ, \nabla^{TZ})\Ch_{H}(E, \nabla^E)\}^{(12)}+(8-2l) \{\hat A(TZ, \nabla^{TZ})\}^{(12)}\right),
\end{split}
\ee
provided $\Tr(R^{TZ})^2=\Tr(BI+R^E)^2.$
\end{theorem}

\section{Elliptic pseudodifferential genera}
In this section, we give an application of the projective elliptic genera by constructing  {\em projective elliptic pseudodifferential genera}  for any projective elliptic pseudodifferential operator. For references on the basics of pseudodifferential operators, see \cite{Hormander, Shubin}. Our construction of elliptic pseudodifferential genera suggests the existence of putative  $S^1$-equivariant  elliptic pseudodifferential operators on loop space that localises to the elliptic pseudodifferential genera, by a formal application of the Atiyah-Segal-Singer localisation theorem, \cite{AS2,AS3}. We also compute the elliptic pseudodifferential genera for some concrete elliptic pseudodifferential operators in this section.

Let $Z$ be a $4r$-dimensional compact oriented smooth manifold. Let $W_3(Z)\in H^3(M, \Z)$ be the third integral Stiefel-Whitney class. 
Fix a projective spin$^c$ structure on $Z$ and let $S^{\pm}_c(TZ)$  be the complex projective spin$^c$ bundle of $TZ$ with twist $W_3(Z)$.  Denote by $S_c(TZ)$ the bundle $S^+_c(TZ)\oplus S^-_c(TZ)$. Let $TZ^{\bot}$ be the stable complement of the tangent bundle $TZ$. Let $S^{\pm}_c(TZ^{\bot})$ be the projective spin$^c$ bundle of $TZ^{\bot}$ with twist $-W_3(Z)$. Denote by $S_c(TZ^{\bot})$ the bundle $S^+_c(TZ^{\bot})\oplus S^-_c(TZ^{\bot})$. Let $\nabla^{S^\pm_c(TZ^{\bot})}$ be projective Hermitian connections on $S^\pm_c(TZ^{\bot})$. Denote by $\nabla^{S_c(TZ^{\bot})}$ the $\Z_2$-graded projective Hermitian connection on $S_c(TZ^{\bot})$.
 
Let $P: C^\infty(F_0)\to C^\infty(F_1)$ be a projective pseudodifferential elliptic operator with $F_0$ and $F_1$ being projective Hermitian vector bundles over $Z$ with twist $H$. Let $\nabla_0$ and $\nabla_1$ be projective Hermitian connections on $F_0, F_1$ respectively. Denote by $\nabla^{\F}$ the $\Z_2$-graded projective Hermitian connection on the bundle $\F=F_0\oplus F_1$. There exists $m\in \Z^+$ and projective complex vector bundle $E$ on $Z$ with twist $H-W_3(Z)$ such that $TZ\oplus TZ^{\bot}\cong Z\times \R^m$ and $\F\otimes S_c(TZ^{\bot})=E^{\oplus 2^m }$. Suppose the rank of $E$ is $l$. Define {\em the first rational Pontryagin class of $P$} by 
$$\mathfrak{p}_1(P):=\mathfrak{p}_1(E)\in H^4(Z, \Q).$$
It is clear that the it is well defined.

Let $\mathbb{S}_\lambda$ be the {\em Schur functor} (c.f. Sec. 6.1 in \cite{FH}). They are indexed by Young diagram $\lambda$ and are functors from the category of vector spaces to itself . It is not hard to see that the Schur functor is a continuous functor (c.f. \cite{A67}) and therefore if $(E, \nabla^E)$ is a vector bundle with connection, then applying the Schur functor gives us a vector bundle with  connection $(\mathbb{S}_\lambda(E), \mathbb{S}_\lambda(\nabla^E))$. If $U, V$ be two vector spaces, the exterior power of a tensor product has the following nice expression via the Schur functors :
\be \Lambda^n(U\otimes V)=\bigoplus\, \mathbb{S}_\lambda(U)\otimes \bS_{\lambda'}(V),\ee
where $\bS_\lambda$ is the Schur functor with $\lambda$ running over all the Young diagram with $n$ cells, at most $\dim (U)$ rows, $\dim (V)$ columns, and $\lambda'$ being the transposed Young diagram. Hence on the projective bundle $\Lambda^n(\F\otimes S_c(TZ^{\bot}) )$, there is a projective Hermitian connection
\be \bigoplus\, \left(\mathbb{S}_\lambda(\nabla^\F)\otimes 1+1\otimes \bS_{\lambda'}(\nabla^{S_c(TZ^{\bot})})\right),\ee
Denote this connection by $\Lambda^n(\nabla^\F, \, \nabla^{S_c(TZ^{\bot})} )$. 

In the following, when we write $\psi(\nabla^\F, \, S_c(TZ^{\bot}) )$, where $\psi$ is certain operations on vector bundles constructed from exterior power, it always means the connections constructed in this way. For instance, 
$$\Theta(\nabla^\F, \, \nabla^{S_c(TZ^{\bot})} )=\bigotimes_{u=1}^\infty
\Lambda_{-q^{u}}(\nabla^\F, \, \nabla^{S_c(TZ^{\bot})} )\otimes \bigotimes_{u=1}^\infty
\Lambda_{-q^{u}}(\overline {\nabla^\F}, \, \overline{\nabla^{S_c(TZ^{\bot})} })$$
is a projective Hermitian connection on the $q$-series with virtual projective bundle coefficients, 
$$\Theta(\F \otimes S_c(TZ^{\bot}))=\bigotimes_{u=1}^\infty\Lambda_{-q^{u}}(\F\otimes S_c(TZ^{\bot}))\otimes \bigotimes_{u=1}^\infty
\Lambda_{-q^{u}}(\overline {\F\otimes S_c(TZ^{\bot})}).$$

Set
\be 
\begin{split}
&\HH(\nabla^\F, \, \nabla^{S_c(TZ^{\bot})} )\\
=&\mathrm{Ch}^{\frac{1}{2^{m+1}}}_{-2^ml(H\!\!-\!\!W_3(Z))}\!\det(\overline {\nabla^\F}, \, \overline{\nabla^{S_c(TZ^{\bot})} })\\
&\cdot\GCh^{\frac{1}{2^{m}}}_{H\!\!-\!\!W_3(Z)}\left((\wedge^{even}(\nabla^\F, \, \nabla^{S_c(TZ^{\bot})} )\!\!-\!\!\wedge^{odd}(\nabla^\F, \, \nabla^{S_c(TZ^{\bot})} ))\otimes\Theta(\nabla^\F, \, \nabla^{S_c(TZ^{\bot})} )\right),\\
\end{split}
\ee
\be 
\begin{split}
&\HH_1(\nabla^\F, \, \nabla^{S_c(TZ^{\bot})} )\\
=&\mathrm{Ch}^{\frac{1}{2^{m+1}}}_{-2^ml(H-W_3(Z))}\det (\overline {\nabla^\F}, \, \overline{\nabla^{S_c(TZ^{\bot})} })\\
&\cdot \GCh^{\frac{1}{2^{m}}}_{H-W_3(Z)}\left((\wedge^{even}(\nabla^\F, \, \nabla^{S_c(TZ^{\bot})} )+\wedge^{odd}(\nabla^\F, \, \nabla^{S_c(TZ^{\bot})} )\otimes\Theta_1(\nabla^\F, \, \nabla^{S_c(TZ^{\bot})} )\right)\\
\end{split}
\ee
and for $i=2,3$
\be 
\HH_i(\nabla^\F, \, \nabla^{S_c(TZ^{\bot})} )=\GCh^{\frac{1}{2^{m}}}_{H-W_3(Z)}\left(\Theta_i(\nabla^\F, \, \nabla^{S_c(TZ^{\bot})} )\right).
\ee

\begin{definition} For the projective elliptic pseudodifferential operator $P$,  define the projective elliptic pseudodifferential genera $Ell(P, \tau)$ and $Ell_i(P, \tau), i=1,2,3$ by the following,
\be \label{PP1}
Ell(P, \tau)=\left(\prod_{j=1}^\infty(1-q^j)\right)^{4r-2l}\cdot \int_{Z} \hat A(Z)\mathrm{Ch}(\Theta(T_\CC Z))\HH(\nabla^\F, \, \nabla^{S_c(TZ^{\bot})} ),
\ee

\be 
Ell_1(P, \tau)=\frac{\left(\prod_{j=1}^\infty(1-q^j)\right)^{4r}}{\left(\prod_{j=1}^\infty(1+q^j)\right)^{2l}}\cdot\int_{Z} \hat A(Z)\mathrm{Ch}(\Theta(T_\CC Z))\HH_1(\nabla^\F, \, \nabla^{S_c(TZ^{\bot})} ),
\ee

\be 
Ell_2(P, \tau)=\frac{\left(\prod_{j=1}^\infty(1-q^j)\right)^{4r}}{\left(\prod_{j=1}^\infty(1-q^{j-1/2})\right)^{2l}}\cdot\int_{Z}\hat A(Z)\mathrm{Ch}(\Theta(T_\CC Z))\HH_2(\nabla^\F, \, \nabla^{S_c(TZ^{\bot})} ),
\ee
and
\be 
Ell_3(P, \tau)=\frac{\left(\prod_{j=1}^\infty(1-q^j)\right)^{4r}}{\left(\prod_{j=1}^\infty(1+q^{j-1/2})\right)^{2l}}\cdot\int_{Z}\hat A(Z)\mathrm{Ch}(\Theta(T_\CC Z))\HH_3(\nabla^\F, \, \nabla^{S_c(TZ^{\bot})} ).
\ee
\end{definition}

\begin{remark} We will see from the next theorem that \newline
(i) the genera for $P$ are well defined;\newline
(ii) $Ell(P, \tau) \in \Q[[q]], \, Ell_1(P, \tau) \in \Q[[q]]$ and $Ell_2(P, \tau) \in \Q[[q^{1/2}]], \, Ell_3(P, \tau) \in \Q[[q^{1/2}]].$ 
\end{remark}

\begin{theorem}\label{promain} (i) $Ell(P, \tau)=PEll(Z, E, \tau), \, Ell_i(P, \tau)=PEll_i(Z,E, \tau), i=1, 2,3; $\newline
(ii)  If $p_1(TZ)=\mathfrak{p}_1(P)$, then $Ell(P, \tau)$ is a modular form of weight $2r$ over $SL(2, \Z)$, $PEll_1(Z,P, \tau)$ is a modular form of weight $2r$ over $\Gamma_0(2)$, $PEll_2(Z,P, \tau)$ is a modular form of weight $2r$ over $\Gamma^0(2)$ and $PEll_3(Z,P, \tau)$ is a modular form of weight $2r$ over $\Gamma_\theta(2)$; moreover, we have 
$$Ell_1(P, -1/\tau)=\tau^{2r}Ell_2(P, \tau), \ \ Ell_2(P, \tau+1)=Ell_3(P, \tau).$$
\end{theorem}
\begin{proof} (i) By the multiplicativity of the operations $\det(V), \wedge^{even}E-\wedge^{odd}E$ and $\Theta(E)$, it is not hard to see that the cohomology class
\be
\begin{split}
&\left[\sqrt{\mathrm{Ch}_{-2^ml(H-W_3(Z))}(\det \bar E)}\, \GCh_{H-W_3(Z)}\left((\wedge^{even}E-\wedge^{odd}E)\otimes\Theta(E)\right)\right]^{2^m}\\
=&\sqrt{\mathrm{Ch}_{-2^ml(H-W_3(Z))}(\det \overline{E^{\oplus 2^m}})}\, \GCh_{H-W_3(Z)}\left((\wedge^{even}(E^{\oplus 2^m})-\wedge^{odd}(E^{\oplus 2^m}))\otimes\Theta(E^{\oplus 2^m})\right)\\
=&\sqrt{\mathrm{Ch}_{-2^ml(H-W_3(Z))}(\det \overline{(\F \otimes S_c(TZ^{\bot})})}\\
&\cdot \GCh_{H-W_3(Z)}\left((\wedge^{even}(\F\otimes S_c(TZ^{\bot}))-\wedge^{odd}(\F\otimes S_c(TZ^{\bot})))\otimes\Theta(\F\otimes S_c(TZ^{\bot}))\right)\\
\end{split}
\ee
is represented by the differential form $(\HH(\nabla^\F, \, \nabla^{S_c(TZ^{\bot})} ))^{2^m}$. Then $Ell(P, \tau)=PEll(Z, E, \tau)$ follows from $(\ref{PP1})$ and $(\ref{P1})$. 
One can similarly use the multiplicativity to prove that $Ell_i(P, \tau)=PEll_i(Z,E, \tau), i=1, 2,3.$

Combining (i) and Theorem \ref{modularity}, (ii) is obtained.  
\end{proof}

In the following, we give two examples of explicit computation of the elliptic pseudodifferential genera. 

\begin{example} \label{spinc} Let $Z$ be a spin$^c$ manifold and $P={\not\! \partial}^+_c$, the spin$^c$ Dirac operator. Then the $E$ corresponding to $P$ is just the trivial complex line bundle $\C$. Hence by (i) in Theorem \ref{promain}, we have 
\be Ell({\not\! \partial}^+_c, \tau)=PEll(Z, \CC, \tau)=0\in \Q[[q]],\ee
\be Ell_1({\not\! \partial}^+_c, \tau)=PEll_1(Z, \CC, \tau)=2W(Z)\in \Q[[q]],\ee
\be Ell_2({\not\! \partial}^+_c, \tau)=PEll_2(Z, \CC, \tau)=W(Z)\in \Q[[q^{1/2}]],\ee and 
\be Ell_3({\not\! \partial}^+_c, \tau)=PEll_3(Z, \CC, \tau)=W(Z)\in \Q[[q^{1/2}]],\ee
where $W(Z)$ is the Witten genus of $Z$. This is similar to looking at the $\hat A$-genus on spin$^c$ manifolds in \cite{MMS06}.
\end{example}

\begin{example} On $\C P^2$, take the standard spin$^c$ structure from the complex structure. let $\F=\mathcal{O}(1)\otimes (\wedge^{ev} T \oplus \wedge^{odd} T)$, where $\mathcal{O}(1)$ is the canonical complex line bundle and $T$ stands for the complex tangent bundle.  Let $P: \mathcal{O}(1)\otimes \wedge^{ev} T\to \mathcal{O}(1)\otimes \wedge^{odd} T$ be an elliptic pseudodifferential operator. Here the $E$ corresponding to $P$ happens to be the honest line bundle $\mathcal{O}(1)$ rather than a projective one.  Let $x\in H^2(\C P^2, \Z)$ be the generator and $z=\frac{x}{2\pi \sqrt{-1}}. $ By (i) in Theorem \ref{promain}, we have

\be Ell(P, \tau)=PEll(\C P^2, \mathcal{O}(1), \tau)=\int_{\C P^2} \left(\frac{z\theta'(0, \tau)}{\theta(z, \tau)}\right)^3\frac{\theta(z, \tau)}{\theta'(0, \tau)}=0\ee
due to $\left(\frac{z\theta'(0, \tau)}{\theta(z, \tau)}\right)^3\frac{\theta(z, \tau)}{\theta'(0, \tau)}$ is an odd function of $z$;
\be 
\begin{split}
Ell_1(P, \tau)
=&PEll_1(\C P^2, \mathcal{O}(1), \tau)\\
=&\int_{\C P^2} \left(\frac{z\theta'(0, \tau)}{\theta(z, \tau)}\right)^3\frac{\theta_1(z, \tau)}{\theta_1(0, \tau)}\\
=&-\frac{1}{8\pi^2}\left(\frac{\theta_1''(0, \tau)}{\theta_1(0, \tau)}-\frac{\theta'''(0, \tau)}{\theta'(0, \tau)} \right)\\
=&2q+\cdots;
\end{split}
\ee
\be
\begin{split}
Ell_2(P, \tau)
=&PEll_2(\C P^2, \mathcal{O}(1), \tau)\\
=&\int_{\C P^2} \left(\frac{z\theta'(0, \tau)}{\theta(z, \tau)}\right)^3\frac{\theta_2(z, \tau)}{\theta_2(0, \tau)}\\
=&-\frac{1}{8\pi^2}\left(\frac{\theta_2''(0, \tau)}{\theta_2(0, \tau)}-\frac{\theta'''(0, \tau)}{\theta'(0, \tau)} \right)\\
=&-\frac{1}{8}-q^{1/2}+\cdots \in \Q[[q^{1/2}]]
\end{split}
\ee
and 
\be 
\begin{split}
Ell_3(P, \tau)
=&PEll_3(\C P^2, \mathcal{O}(1), \tau)\\
=&\int_{\C P^2} \left(\frac{z\theta'(0, \tau)}{\theta(z, \tau)}\right)^3\frac{\theta_3(z, \tau)}{\theta_3(0, \tau)}\\
=&-\frac{1}{8\pi^2}\left(\frac{\theta_3''(0, \tau)}{\theta_3(0, \tau)}-\frac{\theta'''(0, \tau)}{\theta'(0, \tau)} \right)\\
=&-\frac{1}{8}+q^{1/2}+\cdots\in \Q[[q^{1/2}]].
\end{split}
\ee
They are all rational $q$-series rather than integral $q$-series. 
\end{example}


\section*{Appendix. Some number theory preparation}

A general reference for this appendix is \cite{Ch85}.

Let $$ SL_2(\mathbb{Z}):= \left\{\left.\left(\begin{array}{cc}
                                      a&b\\
                                      c&d
                                     \end{array}\right)\right|a,b,c,d\in\mathbb{Z},\ ad-bc=1
                                     \right\}
                                     $$
 as usual be the modular group. Let
$$S=\left(\begin{array}{cc}
      0&-1\\
      1&0
\end{array}\right), \ \ \  T=\left(\begin{array}{cc}
      1&1\\
      0&1
\end{array}\right)$$
be the two generators of $ SL_2(\mathbb{Z})$. Their actions on
$\mathbb{H}$ are given by
$$ S:\tau\rightarrow-\frac{1}{\tau}, \ \ \ T:\tau\rightarrow\tau+1.$$

Let
$$ \Gamma_0(2)=\left\{\left.\left(\begin{array}{cc}
a&b\\
c&d
\end{array}\right)\in SL_2(\mathbb{Z})\right|c\equiv0\ \ (\rm mod \ \ 2)\right\},$$

$$ \Gamma^0(2)=\left\{\left.\left(\begin{array}{cc}
a&b\\
c&d
\end{array}\right)\in SL_2(\mathbb{Z})\right|b\equiv0\ \ (\rm mod \ \ 2)\right\}$$

$$ \Gamma_\theta=\left\{\left.\left(\begin{array}{cc}
a&b\\
c&d
\end{array}\right)\in SL_2(\mathbb{Z})\right|\left(\begin{array}{cc}
a&b\\
c&d
\end{array}\right)\equiv\left(\begin{array}{cc}
1&0\\
0&1
\end{array}\right) \mathrm{or} \left(\begin{array}{cc}
0&1\\
1&0
\end{array}\right)\ \ (\rm mod \ \ 2)\right\}$$
be the three modular subgroups of $SL_2(\mathbb{Z})$. It is known
that the generators of $\Gamma_0(2)$ are $T,ST^2ST$, the generators
of $\Gamma^0(2)$ are $STS,T^2STS$  and the generators of
$\Gamma_\theta$ are $S$, $T^2$. (cf. \cite{Ch85}).

The four Jacobi theta-functions (c.f. \cite{Ch85}) defined by
infinite multiplications are

\be \theta(v,\tau)=2q^{1/8}\sin(\pi v)\prod_{j=1}^\infty[(1-q^j)(1-e^{2\pi \sqrt{-1}v}q^j)(1-e^{-2\pi
\sqrt{-1}v}q^j)], \ee
\be \theta_1(v,\tau)=2q^{1/8}\cos(\pi v)\prod_{j=1}^\infty[(1-q^j)(1+e^{2\pi \sqrt{-1}v}q^j)(1+e^{-2\pi
\sqrt{-1}v}q^j)], \ee
 \be \theta_2(v,\tau)=\prod_{j=1}^\infty[(1-q^j)(1-e^{2\pi \sqrt{-1}v}q^{j-1/2})(1-e^{-2\pi
\sqrt{-1}v}q^{j-1/2})], \ee
\be \theta_3(v,\tau)=\prod_{j=1}^\infty[(1-q^j)(1+e^{2\pi
\sqrt{-1}v}q^{j-1/2})(1+e^{-2\pi \sqrt{-1}v}q^{j-1/2})], \ee where
$q=e^{2\pi \sqrt{-1}\tau}, \tau\in \mathbb{H}$.

They are all holomorphic functions for $(v,\tau)\in \mathbb{C \times
H}$, where $\mathbb{C}$ is the complex plane and $\mathbb{H}$ is the
upper half plane.

Let $\theta^{'}(0,\tau)=\frac{\partial}{\partial
v}\theta(v,\tau)|_{v=0}$. The {\em Jacobi identity} \cite{Ch85},
$$\theta^{'}(0,\tau)=\pi \theta_1(0,\tau)
\theta_2(0,\tau)\theta_3(0,\tau)$$ holds.

The theta functions satisfy the the following
transformation laws (cf. \cite{Ch85}), 
\be 
\theta(v,\tau+1)=e^{\pi \sqrt{-1}\over 4}\theta(v,\tau),\ \ \
\theta\left(v,-{1}/{\tau}\right)={1\over\sqrt{-1}}\left({\tau\over
\sqrt{-1}}\right)^{1/2} e^{\pi\sqrt{-1}\tau v^2}\theta\left(\tau
v,\tau\right)\ ;\ee 
\be \theta_1(v,\tau+1)=e^{\pi \sqrt{-1}\over
4}\theta_1(v,\tau),\ \ \
\theta_1\left(v,-{1}/{\tau}\right)=\left({\tau\over
\sqrt{-1}}\right)^{1/2} e^{\pi\sqrt{-1}\tau v^2}\theta_2(\tau
v,\tau)\ ;\ee 
\be\theta_2(v,\tau+1)=\theta_3(v,\tau),\ \ \
\theta_2\left(v,-{1}/{\tau}\right)=\left({\tau\over
\sqrt{-1}}\right)^{1/2} e^{\pi\sqrt{-1}\tau v^2}\theta_1(\tau
v,\tau)\ ;\ee 
\be\theta_3(v,\tau+1)=\theta_2(v,\tau),\ \ \
\theta_3\left(v,-{1}/{\tau}\right)=\left({\tau\over
\sqrt{-1}}\right)^{1/2} e^{\pi\sqrt{-1}\tau v^2}\theta_3(\tau
v,\tau)\ .\ee

Let $\Gamma$ be a subgroup of $SL_2(\mathbb{Z}).$ A modular form over $\Gamma$ is a holomorphic function $f(\tau)$ on $\mathbb{H}\cup
\{\infty\}$ such that for any
 $$ g=\left(\begin{array}{cc}
             a&b\\
             c&d
             \end{array}\right)\in\Gamma\ ,$$
 the following property holds
 $$f(g\tau):=f\left(\frac{a\tau+b}{c\tau+d}\right)=\chi(g)(c\tau+d)^kf(\tau), $$
 where $\chi:\Gamma\rightarrow\mathbf{C}^*$ is a character of
 $\Gamma$ and $k$ is called the weight of $f$.


\end{document}